\documentclass{amsart}
\usepackage{enumerate}
\usepackage{amssymb,latexsym,amsxtra,amscd,amsfonts}
\usepackage[mathscr]{eucal}
\usepackage{dsfont}
\usepackage{calrsfs}

\newtheorem{thm}{Theorem}[section]
\newtheorem{cor}[thm]{Corollary}
\newtheorem{lem}[thm]{Lemma}
\newtheorem{pr}[thm]{Proposition}

\theoremstyle{definition}

\theoremstyle{remark}
\newtheorem{rem}[thm]{Remark}

\newcounter{obr}[section]

\newcounter{pvv}[section]
\renewcommand{\thepvv}{\thesection.\arabic{pvv}}
{\end{trivlist}}

\newcommand{\saM}{\m_{sa}}
\newcommand{\psaM}{\m_{*sa}}

\newcommand{\m}{\mathcal M}

\newcommand{\B}{\mathcal B}
\newcommand{\N}{\mathds N}
\newcommand{\R}{\mathds R}

\def\Re{\mathop{\rm Re}\nolimits}
\def\Im{\mathop{\rm Im}\nolimits}

\def\id{\mathop{\rm id}\nolimits}
\newcommand{\f}{\ensuremath{\varphi}}

\newcommand{\vNa}{von Neumann algebra}

\def\d{\,{\rm d}}

\title{On Markushevich bases in preduals of von Neumann algebras}

\author{Martin Bohata, Jan Hamhalter and Ond\v{r}ej F.K. Kalenda}

\address{Czech Technical University in Prague, Faculty of Electrical Engineering, Department of Mathematics, Technick\'a 2, 166~27 Prague 6, Czech Republic}
\email{bohata@math.feld.cvut.cz}
\email{hamhalte@math.feld.cvut.cz}
\address{Charles University in Prague, Faculty of Mathematics and Physics, Department of Mathematical Analysis,
Sokolovsk\'a 86, 186~75 Praha 8, Czech Republic}
\email{kalenda@karlin.mff.cuni.cz}

\begin{document}
\begin{abstract}
We prove that the predual of any von Neumann algebra is $1$-Plichko, i.e., it has a countably $1$-norming
Markushevich basis. This answers a question of the third author who proved the same for preduals of semifinite von Neumann algebras. As a corollary we obtain an easier proof of a result of U.~Haagerup that the predual of any von Neumann algebra
enjoys the separable complementation property. We further prove that the self-adjoint part of the predual is $1$-Plichko as well.
\end{abstract}
\thanks{Our research was supported in part by the grant GA\v{C}R P201/12/0290.}
\maketitle

\section{Introduction and main results}

An important tool for the study of nonseparable Banach spaces is a decomposition of the space to some smaller pieces, for example separable subspaces. A decomposition of this type can be done using various kinds of bases or systems of projections.
One of the largest classes of Banach spaces admitting a reasonable decomposition is that of Plichko spaces. The study of this class was initiated by A.~Plichko \cite{plichko82}, later it was investigated using different definitions for example in \cite{val90,val91,dg93}. It appeared to be a common roof for previous search for decompositions of nonseparable spaces in
\cite{lin65,lin66,AL,mer87,AMN} and elsewhere. A detailed survey on this class and some related classes can be found in \cite{survey}. It turned out that this class has several equivalent characterizations. Let us name some of them. We will use the following theorem. 

\smallskip

{\noindent \bf Theorem A. }{\it Let $X$ be a (real or complex) Banach space and let $D\subset X^*$ be a norming linear subspace. Then the following assertions are equivalent.
\begin{itemize}
	\item[(1)] There is a linearly dense set $M\subset X$ such that
	$$D=\{ x^*\in X^*: \{m\in M: x^*(m)\ne 0\}\mbox{ is countable}\}.$$
	\item[(2)] There is a Markushevich basis $(x_\alpha,x^*_\alpha)_{\alpha\in\Gamma}\subset X\times X^*$ such that
	 $$D=\{ x^*\in X^*: \{\alpha\in \Gamma: x^*(x_\alpha)\ne 0\}\mbox{ is countable}\}.$$
	 \item[(3)] There is a system of bounded linear projections $(P_\lambda)_{\lambda\in\Lambda}$ where $\Lambda$ is an up-directed set such that the following conditions are satisfied:
	 \begin{itemize}
	\item[(i)] $P_\lambda X$ is separable for each $\lambda$ and $X=\bigcup_{\lambda\in\Lambda} P_\lambda X$,
	\item[(ii)] $P_\lambda P_\mu=P_\mu P_\lambda =P_\lambda$ whenever $\lambda\le\mu$,
	\item[(iii)] if $(\lambda_n)$ is an increasing sequence in $\Lambda$, it has a supremum $\lambda\in\Lambda$ and $P_\lambda X=\overline{\bigcup_n P_{\lambda_n}X}$,
	\item[(iv)] $P_\lambda P_\mu=P_\mu P_\lambda$ for $\lambda,\mu\in \Lambda$,
	\item[(v)] $D=\bigcup_{\lambda\in\Lambda} P^*_\lambda X^*$.
\end{itemize}
	 \end{itemize}}

\smallskip

Recall that a subspace $D\subset X^*$ is \emph{norming} if $\|x\|_D=\sup\{|x^*(x)|:x^*\in D\cap B_{X^*}\}$ defines an equivalent norm on $X$. If $\|\cdot\|_D=\|\cdot\|$, the subspace $D$ is called \emph{$1$-norming}.
A subspace $D$ satisfying one of the equivalent conditions from Theorem A is called a \emph{$\Sigma$-subspace} of $X^*$. A Banach space admitting a norming $\Sigma$-subspace is said to be \emph{Plichko}. If it admits even a $1$-norming subspace, it is called \emph{$1$-Plichko}. If the dual $X^*$ itself is a $\Sigma$-subspace, $X$ is \emph{weakly Lindel\"of determined} (or, shortly, \emph{WLD}). 
 
Let us comment Theorem A and its proof. The condition (1) is used as a definition of a $\Sigma$-subspace for example in \cite{val-exa}, the definition used in \cite{survey} is easily seen to be equivalent. The implication (2)$\Rightarrow$(1) follows from the definition of a Markushevich basis, the implication (1)$\Rightarrow$(2) is proved in \cite[Lemma 4.19]{survey}. The Markushevich basis from the condition (2) is called \emph{countably norming} (\emph{countably $1$-norming} if $D$ is $1$-norming). This kind of bases were studied among others by A.~Plichko in \cite{plichko82}.

A family of projections satisfying the conditions (i)--(iii) from (3) is called \emph{projectional skeleton}. This notion was introduced by W.~Kubi\'s in \cite{kubisSkeleton}. A projectional skeleton fulfilling moreover the condition (iv) is said to be \emph{commutative}. The condition (v) says that $D$ is the \emph{subspace induced by the respective projectional skeleton}.
The implication (1)$\Rightarrow$(3) is proved in \cite[Proposition 21]{kubisSkeleton}, the converse implication follows from \cite[Theorem 27]{kubisSkeleton}. There are Banach spaces with a projectional skeleton but without a commutative one, see \cite{kubisSkeleton,cuthSimul}. 
 
$1$-Plichko spaces naturally appear in many branches of analysis. Some examples were collected in \cite{val-exa}. They include
spaces $L^1(\mu)$ for an arbitrary non-negative $\sigma$-additive measure $\mu$, order-continuous Banach lattices, the spaces $C(G)$ where $G$ is a compact abelian group and preduals of semifinite von Neumann algebras. It was asked in \cite[Question 7.5]{val-exa} whether the semifiniteness assumption can be omitted. We prove that it is the case. It is the content of the following theorem.
 
\begin{thm}\label{T:main1}
Let $\m$ be any \vNa. Its predual $\m_*$ is then $1$-Plichko. Moreover, $\m_*$ is weakly Lindel\"of determined if and only if $\m$ is $\sigma$-finite. In this case $\m_*$ is even weakly compactly generated.
\end{thm}

Recall that a \vNa{} is \emph{$\sigma$-finite} if any orthogonal family of its projections is countable. The basic setting of \vNa{}s is recalled in Section~\ref{S:vN}. As a corollary we get an alternative proof of the following result.

\begin{cor}[U.~Haagerup, Theorem IX.1 of \cite{GGMS}]
The predual of any \vNa{} enjoys the $1$-separable complementation property. I.e., any separable subspace is contained in a $1$-complemented separable superspace.
\end{cor}

Let us remark that the original proof used very advanced areas of the theory of \vNa{}s. Our proof is more elementary, it follows immediately from the characterization of $1$-Plichko spaces using the condition (3) of Theorem A, together with the observation that the projections can have norm one if $D$ is $1$-norming \cite[Theorem 27]{kubisSkeleton}.

Since the dual of any $C^*$-algebra is a predual of a \vNa{} by \cite[Theorem III.2.4]{takesaki}, we get also positive answers to \cite[Questions 7.6 and 7.7]{val-exa} contained in the following corollary.

\begin{cor} The dual of any $C^*$-algebra is $1$-Plichko. \end{cor}

Further, the following theorem gives a positive answer to \cite[Question 7.3]{val-exa}.

\begin{thm}\label{T:main2}
Let $\m$ be any \vNa{} and denote by $\psaM$ the self-adjoint part of its predual. Then $\psaM$ is $1$-Plichko. Moreover, $\psaM$ is weakly Lindel\"of determined if and only if $\m$ is $\sigma$-finite. In this case $\psaM$ is even weakly compactly generated.
\end{thm}

The paper is organized as follows. In Section~\ref{S:pl} we collect some facts on Plichko spaces and related classes of Banach spaces (WLD spaces, weakly compactly generated spaces). Section~\ref{S:vN} contains basic facts on \vNa{}s and their preduals
and, moreover, several auxilliary results used in the proof of the main theorems. The final section contains the proofs of the main results and some remarks.

\section{Some facts on Plichko spaces}\label{S:pl}

In this section we collect several facts on Plichko spaces and related classes of Banach spaces which will be needed to prove our main results. 

The key tool is a result on $1$-unconditional sums of WLD spaces. Let us first define this kind of sums. Let $X$ be a Banach space and $(X_\lambda)_{\lambda\in \Lambda}$ be an indexed family of closed subsets of $X$. The space $X$ is said to be
the \emph{$1$-unconditional sum} of the family  $(X_\lambda)_{\lambda\in \Lambda}$ if the following three conditions are satisfied.
\begin{itemize}
	\item[(1)] $X_\lambda\cap X_\mu=\{0\}$ whenever $\lambda,\mu\in\Lambda$ are distinct;
	\item[(2)] $\|\sum_{\lambda\in F} x_\lambda\|\le\|\sum_{\lambda\in G} x_\lambda\|$ whenever $F\subset G$ are finite subsets of $\Lambda$ and $x_\lambda\in X_\lambda$ for  $\lambda\in G$;
	\item[(3)] the linear span of $\bigcup_{\lambda\in\Lambda} X_\lambda$ is dense in $X$.
\end{itemize}
Note that the condition (1) follows from the condition (2). However, we prefer to formulate it explicitly, as usually the validity of (1) is used in the proof of (2). The promised result is the following one.

\begin{pr}\label{P:1-uncond} Let $X$ be a Banach space which is the $1$-unconditional sum of a family  $(X_\lambda)_{\lambda\in \Lambda}$ of its closed subspaces. If each $X_\lambda$ is WLD, then $X$ is $1$-Plichko.
Moreover, 
$$\{x^*\in X^*: \{\lambda\in\Lambda: x^*|_{X_\lambda}\ne 0\}\mbox{ is countable}\}$$
is a $1$-norming $\Sigma$-subspace of $X^*$.
\end{pr}

\begin{proof} This result is due to A.~Plichko \cite{plichkomail}. A proof can be found in \cite[Step 3 of the proof of Theorem 6.3]{val-exa}.
\end{proof}

An important subclass of Plichko spaces is that of weakly compactly generated spaces. Let us recall that a Banach space $X$ is said to be \emph{weakly compactly generated} (or, shortly, \emph{WCG}) if there is a weakly compact subset of $X$ whose linear span is dense in $X$. The following proposition summarizes some properties of WCG spaces which we will use in the sequel.

\begin{pr}\label{P:wcg} \ 
\begin{itemize}
	\item[(i)] Any reflexive space (in particular, any Hilbert space) is WCG.
	\item[(ii)] Let $X$ be a complex Banach space. Then $X$ is WCG if and only if the real version of $X$ (i.e., the same space considered as a real space) is WCG.
	\item[(iii)] Let $X$ and $Y$ be two Banach spaces. Suppose that $X$ is WCG and that there is a continuous real-linear operator $T:X\to Y$ with dense range. Then $Y$ is WCG.
	\item[(iv)] Let $X$ be a Banach space and $Y_n$, $n\in\N$, a sequence of closed subspaces of $X$. If each $Y_n$ is WCG and the linear span of $\bigcup_{n\in\N} Y_n$ is dense in $X$, then $X$ is WCG as well.
	\item[(v)] Any WCG space is WLD.
\end{itemize}
\end{pr}

\begin{proof} The assertion (i) is well known and trivial. The assertion (ii) easily follows from the well-known fact that the weak topology of $X$ as a complex space coincides with the weak topology of $X$ as a real space. The assertion (iii) is then 
a consequence of (ii).

(iv) This is well known and easy to see. We include an easy proof for completeness. Let $K_n$ be a weakly compact subset of $Y_n$ whose linear span is dense in $Y_n$. By the uniform boundedness principle the set $K_n$ is bounded, hence we can fix $C_n>0$ such that $\|x\|\le C_n$ for $x\in K_n$. Set $K=\{0\}\cup\bigcup_{n\in\N}\frac1{nC_n}K_n$. Then $K$ is weakly compact
in $X$ and its linear span is dense in $X$.

The assertion (v) is nontrivial but well known. It follows from \cite[Proposition 2]{AL}.
\end{proof}

The following proposition is a special case of the assertion (v) of the previous proposition (due to assertions (i) and (iii)).
But we include it since its proof is short and elementary (unlike the proof of (v)) and we will need only this case.

\begin{pr}\label{P:hg} 
Let $X$ be a Hilbert space, $Y$ a Banach space and $T:X\to Y$ a bounded real-linear operator with dense range. Then $Y$ is WLD.
\end{pr}

\begin{proof} Let us first suppose that $T$ is linear. Fix an orthonormal basis $(e_\lambda)_{\lambda\in\Lambda}$ of $X$ and set $M=\{T(e_\lambda):\lambda\in\Lambda\}$. Then $M$ is clearly linearly dense in $Y$. Moreover, let $y^*\in Y^*$ be arbitrary. For each $\lambda\in \Lambda$ we have $y^*(Te_\lambda)=T^*y^*(e_\lambda)$. Hence
$$\{\lambda\in \Lambda:y^*(Te_\lambda)\ne0\}=\{\lambda\in \Lambda:T^*y^*(e_\lambda)\ne0\}$$
is countable. This shows that $Y^*$ is a $\Sigma$-subspace of itself (it satisfies the condition (1) from Theorem A). 

Now, suppose that $T$ is just real-linear. Consider $X$ and $Y$ as real spaces. Since the real version of a complex Hilbert space is a real Hilbert space, by the linear case we get that $Y$ is WLD as a real space. Fix a set $M$ witnessing the validity of condition (1) from Theorem A. If $Y$ is complex, the same set $M$ witnesses that it is WLD also as a complex space. Indeed,
for any $y^*\in Y^*$ we have
$$\{m\in M: y^*(m)\ne0\}\subset\{m\in M: \Re y^*(m)\ne0 \mbox{ or } \Im y^*(m)\ne0\}$$
which is a countable set.
\end{proof}

\section{Auxilliary results on von Neumann algebras}\label{S:vN}

In this section we collect basic definitions and some results on von Neumann algebras and their preduals which we will use
in the proof of the main results. We start by fixing the basic notation.

Let $H$ be a complex Hilbert space. By $\B(H)$ we denote the algebra of all bounded linear operators on $H$. For a subset $\mathcal{A}\subset\B(H)$ we denote by $\mathcal{A}'$ its \emph{commutant}, i.e., the set of all the operators commuting with all the elements of $\mathcal{A}$. Further, $\m\subset\B(H)$ is a \emph{\vNa} if it is a $*$-subalgebra (i.e., a linear subspace which is closed with respect to composition and  taking the adjoint) which is equal to its double-commutant $\m''$.
Any \vNa{} $\m$ admits a unique predual (see, e.g., \cite[Theorem II.2.6(iii) and Corollary III.3.9]{takesaki}) which we denote by $\m_*$. 

In the sequel we suppose that $H$ is a fixed complex Hilbert space and $\m\subset\B(H)$ a fixed \vNa.

We will need certain standard operators on $\m^*$ (the Banach-space dual of $\m$) which we will denote $A$, $S$, $L_a$ and $R_a$ for $a\in\m$. They are defined as follows.
$$\begin{aligned}
A\f(x)&=\overline{\f(x^*)},\\
S\f(x)&=\frac12(\f(x)+A\f(x))=\frac12(\f(x)+\overline{\f(x^*)}),\\
L_a\f(x)&=\f(ax),\\
R_a\f(x)&=\f(xa)
\end{aligned}$$
for $\f\in\m^*$ and $x\in\m$. Note that $A\f=\f$ if and only if $S\f=\f$. Such functionals are called \emph{self-adjoint} (or \emph{hermitian}). The real Banach space of all the self-adjoint functionals on $\m$ is denoted by $\m^*_{sa}$, the self-adjoint part of $\m_*$ is denoted by $\psaM$.

The following lemma summarizes the basic properties of the above-defined operators:	

\begin{lem}\label{L:oper} \ 
\begin{itemize}
	\item[(i)] The operator $A$ is a conjugate-linear isometry, the operator $S$ is a real-linear projection of norm one.
	\item[(ii)] The operators $L_a$ and $R_a$ are linear and $\|L_a\|\le\|a\|$, $\|R_a\|\le\|a\|$ for any $a\in\m$.
	\item[(iii)] $L_aR_b=R_bL_a$, $L_aL_b=L_{ab}$, $R_aR_b=R_{ba}$ for each $a,b\in\m$.
	\item[(iv)] $AL_a=R_{a^*}A$ and $AR_a=L_{a^*}A$ for each $a\in\m$.
	\item[(v)] The predual $\m_*$ is invariant for operators $A$, $S$, $L_a$ and $R_a$, $a\in\m$.
\end{itemize}
\end{lem}

\begin{proof} The assertions (i)--(iii) are trivial. Let us prove the first equality from assertion (iv). So, for any $a\in\m$, $\f\in\m^*$ and $x\in\m$ we have
$$AL_a\f(x)= \overline{L_a\f(x^*)}=\overline{\f(ax^*)}=\overline{\f((xa^*)^*)}=A\f (xa^*)=R_{a^*}A\f(x).$$
The second equality is analogous. 

Finally, the assertion (v) follows directly from \cite[Theorem 1.7.8]{S}.
\end{proof}   

An element $p$ of a \vNa{} $\m$ is said to be a \emph{projection} if $p=p^*$ and $p^2=p$. It is the case if and only if $p$ is an orthogonal projection. If $p\in\m$ is a projection, then the operators $L_p$ and $R_p$ are clearly linear projections of norm one.

 Following \cite[Definition 5.5.8]{KR1} we call a projection $p\in\m$ \emph{cyclic} if there is $\xi\in H$ such that $\m'\xi=\{a\xi:a\in\m'\}$ is dense in $pH$. Such a vector $\xi$ is then said to be a \emph{generating vector for $p$}.	

\begin{lem}\label{L:cyc1} Let $\m$ be a \vNa{} and $p\in\m$ be a cyclic projection with generating vector $\xi$.
If $x\in\m$ is such that $x\xi=0$, then $xp=0$. \end{lem}

\begin{proof}
 For any $a\in \m'$ we have $0=ax\xi=xa\xi$.
Since $\m'\xi$ is dense in $pH$, we get that $x|_{pH}=0$, i.e., $xp=0$.
\end{proof}

\begin{lem}\label{L:wcg1} Let $\m$ be a \vNa{} and $p\in\m$ be a cyclic projection. Then the spaces $L_p\m_*$ and $R_p\m_*$ are
weakly compactly generated.
\end{lem}

\begin{proof} We will prove the statement for $L_p$. Note that $L_p$ is a linear projection of norm one.
Fix a generating vector $\xi\in H$ for $p$ and  define $\omega(x)=\langle x\xi,\xi\rangle$ for $x\in\m$. Then clearly $\omega\in\m_*$ and, moreover, $\omega\in L_p\m_*$. Indeed, 
$$L_p\omega (x)=\omega(px)=\langle px\xi,\xi\rangle=\langle x\xi,\xi\rangle=\omega(x),$$ where we used that $p^*=p$ and $p\xi=\xi$. 

Further, for $a,b\in M$ set $[[a,b]]=\omega(b^*a)$, the semi-inner product from the GNS construction. Let $H_\xi$ be the resulting Hilbert space (after factorization and completion). Due to Proposition~\ref{P:wcg}(iii), to show that $L_p\m_*$ is WCG it is enough to prove that there is 
a bounded linear mapping $T:H_\xi\to L_p\m_*$ with dense range. To prove that it suffices to construct a linear map $\Phi:\m\to L_p\m_*$ with dense range and such that
$\|\Phi(a)\|\le [[a,a]]^{1/2}$ for $a\in\m$.

The operator $\Phi$ will be defined by the formula $$\Phi(a)=R_a\omega,\quad a\in\m.$$
Then $\Phi(a)\in\m_*$ for any $a\in\m$. Moreover, $\Phi(a)\in L_p\m_*$. Indeed, 
$$L_p\Phi(a)=L_pR_a\omega=R_aL_p\omega=R_a\omega=\Phi(a).$$
It is hence clear that $\Phi$ is a linear mapping from $\m$ to $L_p\m_*$.
Further, for any $a,x\in\m$ we have
$$|\Phi(a)(x)|^2=|R_a\omega(x)|^2=|\omega(xa)|^2\le|\omega(xx^*)|\cdot|\omega(a^*a)|
\le\|x\|^2\cdot[[a,a]].$$ Hence $\|\Phi(a)\|\le [[a,a]]^{1/2}$.

It remains to show that the range of $\Phi$ is dense in $L_p\m_*$. We use Hahn-Banach theorem. Suppose that $x\in \m$ is such that $x$ restricted to the range of $\Phi$ is zero. It means that for each $a\in \m$ we have $$0=\Phi(a)(x)=R_a\omega(x)=\omega(xa)=\langle
xa\xi,\xi\rangle.$$ In particular, by setting $a=x^*$ we get $$0=\langle xx^*\xi,\xi\rangle=\langle x^*\xi,x^*\xi\rangle=\|x^*\xi\|^2.$$
Hence $x^*\xi=0$, so by Lemma~\ref{L:cyc1} $x^*p=0$, hence $px=(x^*p)^*=0$. Hence, given any $\f\in L_p\m_*$ we have
$$\f(x)=L_p\f(x)=\f(px)=0.$$ Hence $x$ restricted to $L_p\m_*$ is zero. This completes the proof.

The proof that $R_p\m_*$ is WCG is analogous. Or, alternatively, it follows using Proposition~\ref{P:wcg}(iii) from the fact that the operator $A$ is a real-linear isometry which maps $L_p\m_*$ onto $R_p\m_*$. Indeed, for any $\f\in L_p\m_*$ we have
$$R_p A\f=A L_p\f=A\f,$$
hence $A\f\in R_p\m_*$ and, similarly, $A\f\in L_p\m_*$ whenever $\f\in R_p\m_*$.
\end{proof}
 
We will use the following known result several times.

\begin{pr}[\cite{KR1}, Proposition 5.5.9]\label{P:rozklad} Let $\m$ be a \vNa{} and $q\in\m$ be a projection. Then there is a family $(p_\lambda)_{\lambda\in\Lambda}$ of mutually orthogonal cyclic projection such that $\sum_{\lambda\in\Lambda}p_\lambda=q$. In particular, there is such a family with sum equal to $1$ (the unit of $\m$).
\end{pr}

\begin{lem}\label{L:spocetnost}
Let $(p_\lambda)_{\lambda\in\Lambda}$ be a family of mutually orthogonal cyclic projection in $\m$. Then for each $x\in\m$ and $\lambda\in\Lambda$ the sets
$$ \{\mu\in \Lambda : p_\lambda x p_\mu\ne0\}\mbox{ and } \{\mu\in \Lambda : p_\mu x p_\lambda\ne0\}$$
are countable.
\end{lem}

\begin{proof} Since $(p_\mu x p_\lambda)^*=p_\lambda x^* p_\mu$, it is enough to prove that the first set is countable for each $x\in \m$ and each $\lambda\in\Lambda$. So, fix $x\in\m$ and $\lambda\in\Lambda$. Let $\xi_\lambda$ be a generating vector for $p_\lambda$ such that $\|\xi_\lambda\|=1$. Suppose that
$A=\{\mu\in\Lambda: p_\lambda x p_\mu\ne 0\}$ is uncountable.
Let $\mu\in A$ be arbitrary, then there is $\eta_\mu\in p_\mu H$ such that $p_\lambda x \eta_\mu\ne 0$. Since this vector belongs to $p_\lambda H$ and $\m'\xi_\lambda$ is dense in $p_\lambda H$, there is $a_\mu\in \m'$ with $\langle p_\lambda x \eta_\mu,a_\mu \xi_\lambda\rangle\ne 0$. Hence
$$0\ne \langle p_\lambda x \eta_\mu,a_\mu \xi_\lambda\rangle = \langle a_\mu^* p_\lambda x \eta_\mu,\xi_\lambda\rangle
=\langle p_\lambda x  a_\mu^*\eta_\mu,\xi_\lambda\rangle.$$
Since $a_\mu^*\eta_\mu\in p_\mu H$ (as $p_\mu H$ is invariant for any element of $\m'$) and it is a nonzero vector, 
one can find $\theta_\mu\in p_\mu H$ such that $\|\theta_\mu\|=1$ and $\langle p_\lambda x  \theta_\mu,\xi_\lambda\rangle>0$.
Hence there is $\delta>0$ such that
$$A_1=\{\mu\in A: \langle p_\lambda x  \theta_\mu,\xi_\lambda\rangle>\delta \}$$
is uncountable. Let $n\in\N$ be arbitrary and $\mu_1,\dots,\mu_n\in A_1$ be distinct. Then
$$n\delta\le\langle p_\lambda x(\sum_{j=1}^n \theta_{\mu_j}),\xi_\lambda \rangle
\le \|p_\lambda x\|\cdot\|\sum_{j=1}^n \theta_{\mu_j}\|=\|p_\lambda x\|\cdot\sqrt n.$$
Since $n\in\N$ is arbitrary, it is a contradiction completing the proof.
\end{proof}

A projection $q\in\m$ is called \emph{$\sigma$-finite} if the algebra $q\m q$ is $\sigma$-finite, i.e., if any orthogonal family of projections smaller that $q$ is countable. (In \cite{KR1} such projections are called \emph{countably decomposable}.)

\begin{pr}\label{P:rozkop}
Let $x\in\m$. Then there is an orthogonal family of $\sigma$-finite projections $(q_j)_{j\in J}$ such that
$$x=\sum_{j\in J} q_j x q_j$$
in the strong operator topology.
\end{pr}

\begin{proof}
Let  $(p_\lambda)_{\lambda\in\Lambda}$ be a family of mutually orthogonal cyclic projection in $\m$ with sum equal to $1$
provided by Proposition~\ref{P:rozklad}. For any $\lambda\in\Lambda$ let 
$$A_1(\lambda)=\{\lambda\}\cup\{\mu\in \Lambda: p_\lambda x p_\mu\ne 0\mbox{ or }p_\mu x p_\lambda\ne 0\}.$$
By Lemma~\ref{L:spocetnost} this set is countable. Further, define for $n\in\N$ by induction sets
$$A_{n+1}(\lambda)=A_n(\lambda)\cup\bigcup\{A_1(\mu):\mu\in A_n(\lambda)\}$$
and, finally,
$$A(\lambda)=\bigcup_{n\in\N} A_n(\lambda).$$
Then $A(\lambda)$ is countable. Moreover, $\lambda\in A(\lambda)$ and for $\lambda_1,\lambda_2\in \Lambda$ either $A(\lambda_1)=A(\lambda_2)$ or $A(\lambda_1)\cap A(\lambda_2)=\emptyset$. Let us introduce on $\Lambda$ the equivalence
$\lambda_1\sim\lambda_2$ if $A(\lambda_1)=A(\lambda_2)$ and let $J$ be the set of all the equivalence classes. For $j\in J$ fix $\lambda\in j$ and set $q_j=\sum_{\mu\in A(\lambda)} p_\mu$. Then $(q_j)_{j\in J}$ is a family of mutually orthogonal projections with sum equal to $1$. Moreover, each $q_j$ is $\sigma$-finite by \cite[Proposition 5.5.19]{KR1}. Hence
$x=\sum_{j\in J} q_j x$. Further, $q_j x=q_j x q_j$ by the construction. This completes the proof.
\end{proof}

\section{Proofs of the main results}

In this section we give the proofs of Theorems~\ref{T:main1} and~\ref{T:main2} using the results of the previous two sections. 

\begin{proof}[Proof of Theorem~\ref{T:main1}.] 
Let $\m$ be any \vNa. By Proposition~\ref{P:rozklad} there is a family $(p_\lambda)_{\lambda\in\Lambda}$ of mutually orthogonal cyclic projections with sum equal to $1$ (the unit of $\m$). By Lemma~\ref{L:wcg1} we know that $L_{p_\lambda}\m_*$ is WCG for each $\lambda\in\Lambda$. We claim that $\m_*$ is the $1$-unconditional sum of the family   $L_{p_\lambda}\m_*$, $\lambda\in\Lambda$. This fact will be proved in three steps:

1. If $\lambda\ne\mu$, then $L_{p_\lambda}\m_*\cap L_{p_\mu}\m_*=\{0\}$. Indeed, if $\f$ is in the intersection, then
$$\f=L_{p_\lambda}\f=L_{p_\lambda}L_{p_\mu}\f=0.$$

2. Let $F_1$ and $F_2$ be finite subsets of $\Lambda$ such that $F_1\subset F_2$ and $\omega_\lambda\in L_{p_\lambda}\m_*$ for $\lambda\in F_2$. Then
$$\begin{aligned}\left\|\sum_{\lambda\in F_1} \omega_\lambda\right\| 
&= \left\|\sum_{\lambda\in F_1} L_{p_\lambda}\left(\sum_{\mu\in F_2} \omega_\mu\right) \right\|
=\left\| \left(\sum_{\lambda\in F_1} L_{p_\lambda}\right)\left(\sum_{\mu\in F_2} \omega_\mu\right) \right\|
\\&=\left\| L_{\sum_{\lambda\in F_1} p_\lambda}\left(\sum_{\mu\in F_2} \omega_\mu\right) \right\|
\le\left\| L_{\sum_{\lambda\in F_1} p_\lambda}\right\|\cdot\left\|\sum_{\mu\in F_2} \omega_\mu\right\|
=\left\|\sum_{\mu\in F_2} \omega_\mu \right\|.\end{aligned}$$

3. The linear span of $\bigcup_{\lambda\in\Lambda}L_{p_\lambda}\m_*$ is dense in $\m_*$. This follows from the Hahn-Banach theorem since, given any nonzero element $x\in\m$, we can find $\lambda\in\Lambda$ such that $p_\lambda x\ne0$ and hence there is $\omega\in\m_*$ with $\omega(p_\lambda x)\ne0$. Then $L_{p_\lambda}\omega(x)=\omega(p_\lambda x)\ne 0$.

Hence, being a $1$-unconditional sum of WCG spaces, $\m_*$ is $1$-Plichko by Proposition~\ref{P:wcg}(v) and Proposition~\ref{P:1-uncond}. Further, if $\m$ is $\sigma$-finite, then $\Lambda$ is countable and hence $\m_*$ is WCG by Proposition~\ref{P:wcg}(iv). 

Finally, suppose that $\m$ is not $\sigma$-finite. Then the index set $\Lambda$ is uncountable due to \cite[Proposition 5.5.19]{KR1}. For each $\lambda\in\Lambda$ fix a unit vector $\xi_\lambda\in p_\lambda H$ and define
$$\omega_\lambda(x)=\langle x\xi_\lambda,\xi_\lambda\rangle, \quad x\in\m.$$
Then $\omega_\lambda\in L_{p_\lambda}\m_*$ (see the beginning of the proof of Lemma~\ref{L:wcg1}) and clearly $\|\omega_\lambda\|=1$ (the norm is attained at $p_\lambda$). For any finite set $F\subset\Lambda$ and any choice of scalars
$c_\lambda$, $\lambda\in F$, we have
$$\|\sum_{\lambda\in F} c_\lambda\omega_\lambda\|=\sum_{\lambda\in F}|c_\lambda|.$$
Indeed, the inequality ``$\le$'' follows from the triangle inequality. To prove the converse fix complex units $\alpha_\lambda$ such that $\alpha_\lambda c_\lambda=|c_\lambda|$ and set $x=\sum_{\lambda\in F}\alpha_\lambda p_\lambda$. Then $x\in\m$, $\|x\|=1$ and
$$\left(\sum_{\lambda\in F} c_\lambda\omega_\lambda\right)(x)=\sum_{\lambda\in F} c_\lambda\omega_\lambda(x)
=\sum_{\lambda\in F} c_\lambda\langle x\xi_\lambda,\xi_\lambda\rangle =\sum_{\lambda\in F}c_\lambda\alpha_\lambda
=\sum_{\lambda\in F}|c_\lambda|.$$
Hence, $\m_*$ contains an isomorphic copy of $\ell_1(\Lambda)$ and thus it is not WLD. (Indeed, $\ell_1(\Lambda)$ is not WLD 
and WLD spaces are stable to taking closed subspaces \cite[Example 4.39]{survey}.)
\end{proof}

The following proposition provides an explicit description of a $1$-norming $\Sigma$-subspace of $\m=(\m_*)^*$. It provides a better insight to the structure of $\m_*$ and, moreover, it will be used in the proof of Theorem~\ref{T:main2}.

\begin{pr}\label{P:Sigma}  
Let $\m$ be a \vNa{} and $(p_\lambda)_{\lambda\in\Lambda}$ be a family of mutually orthogonal cyclic projections with sum equal to $1$. Then
\begin{equation}\label{eq:D}
\begin{aligned}D&=\{x\in \m: \{\lambda\in\Lambda: p_\lambda x\ne 0\}\mbox{ is countable}\}\\&=\{x\in \m: \{\lambda\in\Lambda: xp_\lambda \ne 0\}\mbox{ is countable}\}\end{aligned}\end{equation}
is a $1$-norming $\Sigma$-subspace of $\m=(\m_*)^*$. Moreover, $D$ is a $*$-subalgebra and a two-sided ideal in $\m$ and it can be expressed as
\begin{equation}\label{eq:1}\begin{aligned}
D&=\{x\in\m : \exists q\in\m\mbox{ a $\sigma$-finite projection such that }x=qx \}
\\&=\{x\in\m : \exists q\in\m\mbox{ a $\sigma$-finite projection such that }x=xq \}
\\&=\{x\in\m : \exists q\in\m\mbox{ a $\sigma$-finite projection such that }x=qxq \},
\end{aligned}\end{equation}
hence it does not depend on the concrete choice of the system $(p_\lambda)_{\lambda\in\Lambda}$.
\end{pr}

\begin{proof} By the proof of Theorem~\ref{T:main1} the space $\m_*$ is the $1$-unconditional sum of WCG subspaces $L_{p_\lambda}\m_*$, $\lambda\in\Lambda$. Therefore, Proposition~\ref{P:1-uncond} yields that 
$$D_1=\{x\in \m: \{\lambda\in\Lambda: p_\lambda x\ne 0\}\mbox{ is countable}\}$$
is a $1$-norming $\Sigma$-subspace of $\m=(\m_*)^*$. Similarly, $\m_*$ is the $1$-unconditional sum of WCG subspaces $R_{p_\lambda}\m_*$, $\lambda\in\Lambda$, hence
$$D_2=\{x\in \m: \{\lambda\in\Lambda:x p_\lambda \ne 0\}\mbox{ is countable}\}$$
is also a $1$-norming $\Sigma$-subspace of $\m$. Moreover, $D_1=D_2$ by Lemma~\ref{L:spocetnost}, which completes the proof of the first part.

It is clear that $x^*\in D_2$ whenever $x\in D_1$. Further, if $x\in D_1$ and $a\in\m$, clearly $xa\in D_1$, hence $D_1$ is a right ideal. Similarly, $D_2$ is a left ideal. Since $D=D_1=D_2$ we conclude that $D$ is a $*$-subalgebra and a two-sided ideal in $\m$.

We continue by proving \eqref{eq:1}. Denote the sets appearing on the right-hand side consecutively $D_3,D_4,D_5$.
Let $x\in D=D_1$. Then $C=\{\lambda\in\Lambda:p_\lambda x\ne 0\}$ is countable and hence the projection $p_C=\sum_{\lambda\in C}p_\lambda$ is $\sigma$-finite by \cite[Proposition 5.5.19]{KR1}. Moreover, clearly $p_C x=x$, hence $x\in D_3$. This proves the inclusion $D\subset D_3$. 

To show the converse observe first that any $\sigma$-finite projection belongs to $D$. Indeed, suppose that $q\in\m$ is a $\sigma$-finite projection. By Proposition~\ref{P:rozklad} there is a sequence $(q_n)$ of mutually orthogonal cyclic projections such that $q=\sum_{n\in\N} q_n$. Let $\xi_n$ be a generating vector for $q_n$. If $\lambda\in\Lambda$ is such that $p_\lambda q\ne0$, then there is $n\in\N$ such that $p_\lambda q_n\ne 0$. By Lemma~\ref{L:cyc1} it follows that $p_\lambda \xi_n\ne 0$. Since the projections $p_\lambda$ are mutually orthogonal, for given $n\in\N$ there can be only countably many $\lambda$ with $p_\lambda\xi_n\ne0$. Therefore, $p_\lambda q\ne 0$ only for countably many $\lambda\in\Lambda$. In other words, $q\in D$. Since $D$ is an ideal, $qx\in D$ whenever $x\in\m$. It follows that $D_3\subset D$, hence $D=D_3$.

We continue by observing that $x\in D_4$ if and only if $x^*\in D_3$. Since $D_3=D$ and $D$ is a $*$-subalgebra, we infer $D=D_4$.

To complete the proof it is enough to observe that $D_3\cap D_4 = D_5$.
Indeed, the inclusion $\supset$ is obvious. To show the converse one, fix $x\in D_3\cap D_4$. Then $x=q_1x=xq_2$ for some $\sigma$-finite projections $q_1,q_2$. Let $q=q_1\vee q_2$ be the projection whose range is
the closed linear span of $q_1H\cup q_2H$. Then $q$ is $\sigma$-finite (cf. \cite[Exercise 5.7.45]{KR1}) and $x=qxq$,  hence $x\in D_5$. 
\end{proof}

The main part of Theorem~\ref{T:main2} follows from the following proposition.

\begin{pr}\label{P:main2} Let $\m$ be a \vNa{} and $\saM$ denote its self-adjoint part. The operator
$\Psi:\saM\to(\psaM)^*$ defined by
$$\Psi(x)(\omega)=\omega(x),\quad x\in\saM,\omega\in\psaM$$
is an onto isometry. Moreover, if we set
$$D_{sa}=\{x\in \saM : \exists q\in\m\mbox{ a $\sigma$-finite projection such that }x=qxq \},$$
then $\Psi(D_{sa})$ is a $1$-norming $\Sigma$-subspace of $(\psaM)^*$.
\end{pr}

\begin{proof} 
It is clear that $\Psi$ is a linear operator between the real Banach spaces $\saM$ and $(\psaM)^*$
and that $\|\Psi(x)\|\le\|x\|$ for each $x\in\saM$. Moreover, $\Psi$ is an isometry due to the facts
that
$$\|x\|=\sup\{|\langle x\xi,\xi\rangle| : \xi\in H,\|\xi\|\le 1\}, \quad x\in\saM,$$
and that the functional $a\mapsto\langle a\xi,\xi\rangle$ belongs to $\psaM$ and has norm at most $\|\xi\|^2$.
It remains to show that $\Psi$ is onto. So, let $\f\in(\psaM)^*$. By the Hahn-Banach theorem it can be extended to a continuous real-valued real-linear functional $\f_1$ on $\m_*$. Then there is a complex linear functional $\f_2$ on $\m_*$ such that
$\f_1(\omega)=\Re \f_2(\omega)$ for $\omega\in\m_*$. Since the dual to $\m_*$ is $\m$, $\f_2$ is represented by some 
$a\in\m$. Then $a=x+iy$ for $x,y\in\saM$. Then for any $\omega\in\psaM$ we have
$$\f(\omega)=\f_1(\omega)=\Re\f_2(\omega)=\Re\omega(a)=\omega(x),$$
in other words $\f=\Psi(x)$.

Further, recall that
$$\psaM=\{\omega\in \m_* : \omega(x)\in\R\mbox{ for each }x\in\saM\}.$$
It follows from Proposition~\ref{P:rozkop} that
\begin{equation}
\label{eq:m*}
\begin{split}
\psaM=\{\omega\in\m_*: \omega(qxq)\in\R\mbox{ for each }x\in\saM\\ \mbox{ and each $\sigma$-finite projection }q\in\m\}.
\end{split}
\end{equation}

Let $D$ be the $1$-norming $\Sigma$-subspace of $\m=(\m_*)^*$ described in Proposition~\ref{P:Sigma}. Let $(\m_*)_R$ denote the Banach space $\m_*$ considered as a real space and let $(\m_*)_R^*$ denote its dual. Let
$$D_R=\{ \omega\mapsto\Re \omega(x) : x\in D\}.$$
Then $D_R$ is a $1$-norming $\Sigma$-subspace of $(\m_*)_R^*$ by \cite[Proposition 3.4]{kal-complex}. Moreover, if $\omega\in\m_*$, $x\in\saM$ and $q$ is a projection, then $\omega(qxq)\in\R$ if and only if $\Re\omega(iqxq)=0$.
Thus
\begin{equation*}\begin{split}\psaM=\{\omega\in\m_*: \Re \omega(iqxq)=0\mbox{ for each }x\in\saM\\ 
\mbox{ and each $\sigma$-finite projection }q\in\m\}.\end{split}\end{equation*}
Since $iqxq=q(ix)q\in D$ for each $x\in\m$ and each $\sigma$-finite projection $q\in\m$, the functional $\omega\mapsto \Re\omega(iqxq)$ belongs in this case to $D_R$. It follows that $\psaM$ is a $\sigma(\psaM,D_R)$-closed linear subspace of
$(\m_*)_R$. ($\sigma(\psaM,D_R)$ denotes the weak topology on $\psaM$ induced by $D_R$.) It follows from \cite[Theorem 4.38]{survey} that
$$D_0=\{ \f|_{\psaM} : \f\in  D_R\}$$
is a $1$-norming $\Sigma$-subspace of $(\psaM)^*$. It remains to verify that $D_0=\Psi(D_{sa})$.

Let $x\in D_{sa}$. Then $x\in D\cap \saM$. In particular, for any $\omega\in\psaM$ we have 
$$\Re\omega(x)=\omega(x)=\Psi(x)(\omega),$$
hence $\Psi(x)\in D_0$. Conversely, let $\f\in D_0$. Then there is $\f_1\in D_R$ with $\f=\f_1|_{\psaM}$. Further, there is $a\in D$ such that $\f_1(\omega)=\Re\omega(a)$ for $\omega\in\m_*$. Then $a=x+iy$ with $x,y\in\saM$. Since $a^*\in D$ as well, clearly $x,y\in D$. Hence $x,y\in D_{sa}$. Moreover, for $\omega\in\psaM$ we have
$$\f(\omega)=\Re\omega(a)=\omega(x)=\Psi(x)(\omega),$$
hence $\f\in\Psi(D_{sa})$. This completes the proof.
\end{proof}

\begin{proof}[Proof of Theorem~\ref{T:main2}]
The space $\psaM$ is $1$-Plichko by Proposition~\ref{P:main2}.
Further, if $\m$ is $\sigma$-finite, $\m_*$ is WCG by Theorem~\ref{T:main1}. Moreover, $\psaM$ is the image of $\m_*$ by the real-linear projection $S$, hence $\psaM$ is WCG by Proposition~\ref{P:wcg}(iii). 

Finally, suppose that $\m$ is not $\sigma$-finite. Let $(\omega_\lambda)_{\lambda\in\Lambda}$ be the uncountable family in $\m_*$ constructed at the end of the proof of Theorem~\ref{T:main1}. It is clear that $\omega_\lambda\in\psaM$ for any $\lambda\in\Lambda$ and that the closed linear span of this family in the real Banach space $\psaM$ is isometric to the real version of the space $\ell_1(\Lambda)$ and hence $\psaM$ is not WLD.
\end{proof}

\begin{rem}
(1) We proved that $\m_*$ is $1$-Plichko since it is the $1$-unconditional sum of WCG subspaces. To get the result we used the classical but highly nontrivial assertion (v) of Proposition~\ref{P:wcg}. It is possible to give a more elementary proof
using Proposition~\ref{P:hg}. Indeed, by the proof of Lemma~\ref{L:wcg1} the spaces $L_{p_\lambda}\m_*$ and $R_{p_\lambda}\m_*$
satisfy the assumptions of Proposition~\ref{P:hg} in place of $Y$, hence it easily follows that they are WLD.

(2) Proposition~\ref{P:Sigma} shows that there is a canonical $1$-norming $\Sigma$-subspace of $\m=(\m_*)^*$. However, there can be many different (non-canonical) $1$-norming $\Sigma$-subspaces, cf. \cite[Example 6.9]{survey} where this is studied for the space $\ell_1(\Gamma)$. However, there is a unique $1$-norming $\Sigma$-subspace which is a two-sided ideal. This is proved in the following proposition.
\end{rem}

\begin{pr} Let $S$ be a $1$-norming $\Sigma$-subspace of $\m=(\m_*)^*$ which is a two-sided ideal in $\m$. Then $S=D$ where $D$ is the $\Sigma$-subspace described in Proposition~\ref{P:Sigma}.
\end{pr}

\begin{proof} 
Being a $\Sigma$-subspace, $S$ is countably weak$^*$-closed, i.e., 
\begin{equation}
\label{eq:S1}
\overline{A}^{w^*}\subset S\mbox{ for each }A\subset S\mbox{ countable.}
\end{equation}
Indeed, it easily follows from the condition (1) of Theorem A. In particular, $S$ is norm-closed, hence it is a $C^*$-subalgebra of $\m$ \cite[Corollary 4.2.10]{KR1}. In particular, the continuous functional calculus works in $S$, i.e. $f(x)\in S$ whever $x\in S$ is self-adjoint and $f:\R\to \R$ is a continuous function with $f(0)=0$. Further, we even have
\begin{equation}
\label{eq:S2}
\begin{split}
f(x)\in S \mbox{ whenever }x\in S\cap \saM\mbox{ and }f:\R\to\R \mbox{ is a bounded function}\\\mbox{of the first Baire class with }f(0)=0.
\end{split}
\end{equation}
Indeed, let $f$ be such a  function. Then there is a uniformly bounded sequence of continuous functions $f_n:\R\to\R$ with $f_n(0)=0$ pointwise converging to $f$. Given any self-adjoint $x\in S$, we have
$f_n(x)\in S$ as well and, moreover, $f_n(x)\to f(x)$ in the weak operator topology. This topology coincides with the weak$^*$-one on bounded sets, hence $f_n(x)\to f(x)$ in the weak$^*$ topology, hence $f(x)\in S$ as well by \eqref{eq:S1}.

We continue by showing that any cyclic projection belongs to $S$. So, let $p\in\m$ be a cyclic projection and $\xi\in H$ a generating vector for $p$ of norm one. Set $\omega(x)=\langle x\xi,\xi\rangle$ for $x\in\m$. Then $\omega$ is a normal state on $\m$. In particular, $\omega\in\psaM$ and $\|\omega\|=1$. Since $S$ is $1$-norming and $S\cap B_\m$ is weak$^*$ countably compact (by \eqref{eq:S1}), there is some $a\in S\cap B_\m$ with $\omega(a)=1$. Since $\omega$ is selfadjoint, we have $\omega(a^*)=1$ as well, hence $b=\frac12(a+a^*)$ is a selfadjoint element of $S\cap B_\m$ with $\omega(b)=1$.

Set $q=\chi_{\R\setminus\{0\}}(b)$. Since $\chi_{\R\setminus\{0\}}$ is of the first Baire class, $q\in S$ by \eqref{eq:S2}. Further, $q$ is clearly a projection. It follows from the properties of the function calculus that $q$ commutes with $b$ and
that
$$qb=\chi_{\R\setminus\{0\}}(b)\id(b)=(\chi_{\R\setminus\{0\}}\cdot\id)(b)=\id(b)=b,$$
hence $b=qbq$. Since
$$1=\omega(b)=\langle b\xi,\xi\rangle,$$
necessarily $b\xi=\xi$, hence $\xi$ belongs to the range of $b$ and so also to the range of $q$. Thus $q\xi=\xi$, hence $(1-q)\xi=0$, so $(1-q)p=0$ by Lemma~\ref{L:cyc1}, hence $p=qp$ and we conclude $p\in S$ (since $S$ is an ideal).

It follows that $S$ contains all cyclic projections and hence all $\sigma$-finite projections (by \eqref{eq:S1} and Proposition~\ref{P:rozklad}. Since $S$ is an ideal, it follows from the description of $D$ in Proposition~\ref{P:Sigma} that $D\subset S$. Hence $D=S$ by \cite[Lemma 2]{almostwld}.
\end{proof}


\end{document}